\documentclass[11pt]{amsart}

\usepackage{fullpage}
\usepackage{amsfonts,color}
\usepackage{graphicx}
\usepackage{epstopdf}

\usepackage{tikz, array, amsmath, amssymb, amsthm,verbatim}
\usetikzlibrary{decorations.pathmorphing,calc,trees,positioning,arrows,fit,shapes}
\usetikzlibrary{decorations.markings}
\usetikzlibrary{backgrounds}

\newcommand{\precequals}{precequals}  

\newcommand{\sgn}{\mathop{\mathrm{sgn}}}
\newcommand{\G}{\Pi}
\newcommand{\cc}{\sigma}
\newcommand{\C}{C}
\newcommand{\Sym}{\mathfrak{S}}
\newcommand{\Pn}{\mathcal{P}_n}
\newcommand{\Pthree}{\mathcal{P}_3}
\newtheorem{theorem}{Theorem}

\newtheorem{conjecture}[theorem]{Conjecture}
\newtheorem{corollary}[theorem]{Corollary}

\newtheorem{lemma}[theorem]{Lemma}

\newtheorem{proposition}[theorem]{Proposition}

\theoremstyle{plain}

\theoremstyle{definition}
\newtheorem{definition}[theorem]{Definition}
\newtheorem{example}[theorem]{Example}

\newcommand{\D}{\mathcal{D}}
\DeclareMathOperator{\perm}{perm}

\begin{document}

\pagestyle{plain}
\title{A Terrible Expansion of the Determinant}
\author{Erik Insko}
\address{Department of Mathematics\\
Florida Gulf Coast University \\ Fort Myers, FL 33965}
\email{einsko@fgcu.edu}
\author{Katie Johnson}
\address{Department of Mathematics\\
Florida Gulf Coast University \\ Fort Myers, FL 33965}
\email{kjohnson@fgcu.edu}
\author{Shaun Sullivan} 
\address{Department of Mathematics\\
Florida Gulf Coast University \\ Fort Myers, FL 33965}
\email{ssullivan@fgcu.edu}

\subjclass[2010]{ 05A18, 05A40,	05E45    }
\keywords{determinant, permutahedron, umbral calculus, set partition}

\date{\today}

\begin{abstract}
 From a transfer formula in multivariate finite operator calculus, comes an expansion for the determinant similar to Ryser's 
formula for the permanent.  Although this one contains many more terms 
than the usual determinant formula.  To prove it, we consider the poset of ordered partitions, properties of the permutahedron, and some good old fashioned combinatorial techniques.  

\end{abstract}

\maketitle

\section{Introduction} \label{sec:Intro}

One of the foundational concepts of linear algebra is the determinant.  At the most basic level, this matrix parameter is celebrated for its intricate ties to the set of 
eigenvalues and as a similarity invariant.  However, the determinant still surprises us as the solution to a varying array of problems.

In addition to solving systems of linear equations and performing a change of variables in calculus, the determinant can help us count!  Benjamin and Cameron \cite{BC} recently showed 
the determinant will calculate the number of nonintersecting $n$-paths in certain nonpermutable digraphs, where an $n$-path is a set of $n$ paths from $n$ distinct source vertices to $n$ 
distinct sink vertices. In fact, the permanent will count the number of all $n$-paths. 

The determinant of a matrix can be found recursively, as an alternating sum of minors.  Often the determinant of an $n\times n$ matrix $A$ is defined compactly using the Leibniz formula, 
precisely \[\det(A)= \sum_{\sigma\in S_n} \sgn(\sigma) \prod_{i=1}^n a_{i,\sigma_i}.\]  Similarly, the permanent can be defined as a sum over subsets of $[n]:=\{1,2, \ldots, n\}$ using Ryser's 
formula \cite{HJ} \[\perm(A)= \sum_{S\subset[n]} (-1)^{n-|S|} \prod_{i=1}^n \sum_{j\in S} a_{i,j}.\] 

In this paper, we prove a much messier expansion of the determinant by instead indexing our terms using the set of ordered partitions of $[n]$.  Aptly, we call this the 
\emph{terrible expansion of the determinant.}  This expansion is analogous to Ryser's formula for the permanent. 
Section \ref{sec:Hardstuff} explains the origins of this expansion as it relates to multivariate finite operator calculus, a 
branch of mathematics that has proven useful in enumerating ballot (generalized Dyck) paths containing certain patterns
\cite{NiederhausenSullivan,sullivan}.

Before stating the formula for our expansion of the determinant, we introduce it with two fundamental examples.  
When  $n=2$ we see that Equation \eqref{eqn:2by2} provides the following expansion for the determinant.  
\begin{eqnarray}
\left|\begin{array}{cc}
a_{11} & a_{12} \\   \label{eqn:2by2}
a_{21} & a_{22} \\
\end{array}\right|&=&a_{11}(a_{12}+a_{22})+a_{22}(a_{11}+a_{21}) -(a_{11}+a_{21})(a_{12}+a_{22})\\ \notag
\end{eqnarray}
Likewise when $n=3$ we get the following expansion:  
\begin{eqnarray}
\notag|A|&=&a_{11}(a_{12}+a_{22})(a_{13}+a_{23}+a_{33})+a_{11}(a_{13}+a_{33})(a_{12}+a_{22}+a_{32}) \\\notag
&&+a_{22}(a_{11}+a_{21})(a_{13}+a_{23}+a_{33})+a_{22}(a_{23}+a_{33})(a_{11}+a_{21}+a_{31}) \\\notag
&&+a_{33}(a_{11}+a_{31})(a_{12}+a_{22}+a_{32})+a_{33}(a_{22}+a_{32})(a_{11}+a_{21}+a_{31}) \\\notag
&&-a_{11}(a_{12}+a_{22}+a_{32})(a_{13}+a_{23}+a_{33})-a_{22}(a_{11}+a_{21}+a_{31})(a_{13}+a_{23}+a_{33}) \\\notag
&&-a_{33}(a_{11}+a_{21}+a_{31})(a_{12}+a_{22}+a_{32})-(a_{11}+a_{21})(a_{12}+a_{22})(a_{13}+a_{23}+a_{33}) \\\notag
&&-(a_{11}+a_{31})(a_{13}+a_{33})(a_{12}+a_{22}+a_{32})-(a_{22}+a_{32})(a_{23}+a_{33})(a_{11}+a_{21}+a_{31}) \\\notag
&&+(a_{11}+a_{21}+a_{31})(a_{12}+a_{22}+a_{32})(a_{13}+a_{23}+a_{33}).\\\notag
\end{eqnarray}

Our main theorem gives a general description of the terrible expansion of the determinant. 

\begin{theorem}  \label{thm:terrible}
Let $A=\left(a_{ij}\right)_{n\times n}$.  The following formula is an expansion for the determinant of $A$:
\begin{equation}\label{eqn:main}
\det(A)=\sum\limits_{B\vdash[n]}(-1)^{n-|B|}\prod\limits_{\beta_k \in B}\prod\limits_{j\in\beta_k }\sum\limits_{i\in\beta^{\prime}_k}a_{ij},
\end{equation} 
where the outer summation runs over all ordered partitions $B= (\beta_1,\beta_2, \ldots, \beta_r)$ of the set $[n]$  and
the inner summation runs over all integers $i$ in the union of first $k$ parts $\beta_k^{\prime}=\bigcup_{j=1}^k\beta_j$ of 
the partition $B$.
\end{theorem}

The following example serves to clarify the notation in Theorem \ref{thm:terrible}.
\begin{example} \label{example:variety}
If \[B=(\beta_1,\beta_2,\beta_3)=\left(\{2\}, \{1,3\}, \{4,5\}\right),\] 
then \[B' = (\beta_1^{\prime},\beta_2^{\prime},\beta_3^{\prime})=\left(\{2\}, \{1,2,3\}, \{1,2,3,4,5\}\right),\] and the 
corresponding expression in Equation \eqref{eqn:main} for the ordered partition $B$ is
\[
a_{22}(a_{11}+a_{21}+a_{31})(a_{13}+a_{23}+a_{33})(a_{14}+a_{24}+a_{34}+a_{44}+a_{54})(a_{15}+a_{25}+a_{35}+a_{45}+a_{55}).
\]
notice that the first index runs through $B'$, 
while the second index runs through the partition $B$. 
\end{example}

The rest of this paper proceeds as follows:
In Section \ref{sec:flat}, we analyze the functions $f:[n]\to [n]$ indexing the terms in the terrible expansion.  
 After setting notation and proving a few fundamental lemmas, we end that section with Corollary \ref{corollary:permutation},
 which proves that when $f:[n] \rightarrow [n]$ is a bijective function, or a permutation, then the 
coefficients $c_f =sgn(f)$. In other words,
they are precisely the nonzero coefficients appearing in the determinant. 
In Section \ref{sec:partitions}, we study the poset of ordered partitions and identify the importance of singleton partitions so 
that we 
 can formulate our problem in more geometric terms as Euler characteristics 
 of convex polytopes relating to the permutahedron. 
In Section \ref{sec:permutahedron}, we prove that $c_f = 0$ for all non-bijective functions $f:[n] \to [n]$ by analyzing Euler 
characteristics of subsets of the permutahedron.
 This proves that the terrible expansion does indeed give a formula for the determinant.  
In Section \ref{sec:Hardstuff}, we give an extremely brief introduction to multivariate finite operator calculus, and state a 
more general open conjecture that motivated this paper.

\section{Flattening Functions} \label{sec:flat}
Upon expanding the expression in Theorem \ref{thm:terrible}, many terms will cancel. 
In this section, we set up the groundwork to keep track of each of the terms and show how they cancel. 
Now consider any function $f:[n] \to [n]$ and define the monomial
 \[
a_f:=\prod\limits_{j=1}^na_{f(j),j}.
\]   
Expanding the terms in Equation \eqref{eqn:main} results in a sum of the form
\begin{equation}\label{eqn:mainf}
\sum\limits_{B\vdash[n]}(-1)^{n-|B|}\prod\limits_{\beta_k \in B}\prod\limits_{j\in\beta_k 
}\sum\limits_{i\in\beta^{\prime}_k}a_{ij} =  \sum_{f} c_f a_f,
\end{equation} 
where each term $c_fa_f$ corresponds to a function from the set $[n]$ to itself.
For instance, in Equation \eqref{eqn:2by2} the four functions $f_i: [2] \to [2]$ are \[ f_1(1)=1, \  f_1(2)=1; \ \ f_2(1) = 1, \ f_2(2)=2; \ \ f_3(1) = 2,  \ f_3(2)=2; \text{ and }  f_4(1) =2, \ f_4(2)=1. \] 
The terms corresponding to each function are labeled below
 \begin{align*} 
 a_{11}(a_{12}+a_{22})+ a_{22}(a_{11}+a_{21}) -&(a_{11}+a_{21})(a_{12}+a_{22})  \\
 =&  (a_{11}a_{12}-a_{11}a_{12}) + (a_{11}a_{22}+ a_{11}a_{22}-a_{11}a_{22}) + \\ &(a_{21}a_{22}-a_{21}a_{22})+(-a_{12}a_{21})    \\
=& 0+a_{11}a_{22} +0 - a_{12}a_{21} \\       
 =& c_{f_1}a_{f_1} + c_{f_2}a_{f_2}+c_{f_3}a_{f_3}+c_{f_4}a_{f_4},\end{align*}
 where first we expand the terms and then we simplify. Hence we see that $c_{f_1}=0$, $c_{f_2}= 1$, $c_{f_3}=0$, and $c_{f_4} = -1$. 
                                                                                 
Thus, the goal of this paper is to combinatorially identify the coefficients $c_f$
for each such function, and show that they agree with the coefficients of $a_f$ in the determinant.  To do so, we must identify the set $S_f=\{ B \vdash [n]:  a_f \text{ appears as a summand of the product indexed by } B\}$.  
The following definition and lemma describe criteria for when an ordered partition $B \vdash [n]$ appears in $S_f$.  

\begin{definition}
Let $B= ( \beta_1, \beta_2, \ldots, \beta_r ) $ be an ordered partition of $[n]$.  For each $i\in [n]$, define $\beta(i)=k$ if $i\in\beta_k$. We say $i$ \textbf{\precequals} $j$ in $B$, denoted $i \preceq j$,
if $i$ appears in an earlier part or the same part as $j$ in the ordered partition $B$, i.e.,  \[ \beta(i)\leq\beta(j) .\]

\end{definition}

\begin{lemma}\label{lemma:obvious}
Let $B \vdash [n]$ be an ordered partition of $[n]$. The term $a_f$ appears 
in the product $ \prod\limits_{\beta_k \in B}\prod\limits_{j\in\beta_k }\sum\limits_{i\in\beta^{\prime}_k}a_{ij} $ iff  
$f(j) \preceq j$ in $B$ for all $1 \leq j \leq n$. \end{lemma}

\begin{proof}
Let $B \vdash [n]$ be an ordered partition of $[n]$.  
Since $\beta'_k = \displaystyle \cup_{j=1}^k \beta_j$, we see that $a_f$ appears as a term in the product $\prod\limits_{\beta_k \in B}\prod\limits_{j\in\beta_k }\sum\limits_{i\in\beta^{\prime}_k}a_{ij}$
precisely when $i=f(j) \preceq j$ in the ordered partition $B$ for all $1 \leq j \leq n$.  
\end{proof}

Now that Lemma \ref{lemma:obvious} has established that an ordered partition $B$ appears in the set $S_f$ if and only if 
$f(j)$ precequals $j$ in $B$ for all integers $1 \leq j \leq n$, we are ready to analyze which functions $f:[n] \to [n]$ correspond to nonzero
terms in the terrible expansion.  
To do so, we start by introducing some notation regarding the structure of such functions.

\begin{definition}\label{def:cyclic}
Let $f: [n] \to [n]$ be any function. We say that the function $f$ is \emph{acyclic} if $f^k(i)=i$ implies $k=1$ for all $i\in 
[n]$.  Otherwise, we say $f$ contains a \emph{cycle }.
\end{definition}

Henceforth we will prefer to work with acyclic functions.  The next definition, describes how each function $f:[n]\rightarrow 
[n]$ which contains a cycle 
can be simplified or \emph{flattened}
to an acyclic function $\bar{f}$ on a different set of elements. 

\begin{definition}\label{def:flattened}
Let $f: [n] \to [n]$ be any function, $\C_f=\{\cc_1, \cc_2, \ldots, \cc_r\}$ represent the cycles 
of $f$, $N_f\subseteq[n]$ be the elements not in a cycle, and $D_f=\C_f\cup N_f$. 
Define the \emph{flattened} function $\overline{f}:D_f\to D_f$ as follows 

\[
\overline{f}(i)=\begin{cases}
f(i) & \text{if } f(i) \text{ is not in a cycle of } f \\
\cc_j & \mbox{if } f(i) \mbox{ belongs to the cycle } \cc_j \\
i & \mbox{if } i\in \C_f \\
\end{cases}
.\] 
\end{definition}

Intuitively, $\overline{f}$ acts just like $f$, but shrinks each cycle of $f$ to a fixed point, and thus it is an acyclic 
function. 
The following example
illustrates Definitions \ref{def:cyclic} and \ref{def:flattened}.

\begin{example}
Let $f:[6] \to [6]$ be defined by $f(1) = 1, f(2) = 3, f(3)=2$, $f(4) = 3$, $f(5) = 6$, and $f(6) = 5$.  Then $\C_f = \{\cc_1, 
\cc_2 \}$ where $\cc_1 = (23)$ and $c_2 = (56)$. 
The sets $N_f = \{ 1,4 \}$ and $D_f = N_f \cup \C_f = \{ \cc_1, \cc_2, 1,4 \}$.  
The function $\overline{f} : D_f \to D_f$ is defined by $\overline{f}(\cc_1) = \cc_1$, $\overline{f}(\cc_2) = \cc_2$, 
$\overline{f}(1) =1$, and $\overline{f}(4) = \cc_1$. 
\end{example}

The following lemma shows that we can reduce the problem of calculating the coefficients $c_f$ in the terrible expansion, to that 
of calculating
the coefficients $c_{\bar{f}}$ corresponding to acyclic functions.

\begin{lemma}\label{lemma:flattenorbits}
If $f: [n] \to [n]$ is any function, then the coefficients $c_f$ and $c_{\overline{f}}$ are related by the equation $c_f=(-1)^{n-|D_f|} c_{\overline{f}}$.  
\end{lemma}
\begin{proof}
	Let $S_f : =\{ B \vdash [n]:  a_f \text{ appears as a summand of the product indexed by } B\}$, and
	similarly let $S_{\overline{f}} : =\{ A \vdash D_f:  a_{\overline{f}} \text{ appears as a summand of the product indexed 
by } A\}$.
  
If $B$ is an ordered partition for which $a_f$ appears as a summand, then Lemma \ref{lemma:obvious} implies that for each cycle \[ 
\cc_j=(i,f(i), f^2(i), \ldots, f^k(i))\] of $f$,
the following precedence relation must hold \[ i \preceq f^k(i) \preceq f^{k-1}(i) \preceq \cdots \preceq f(i) \preceq i \] in the ordered partition $B$.
 Therefore $c_j$ must be contained in the same part $\beta$ of the ordered partition $B$. From here, it is clear to see that $S_f$ 
is equivalent to the set of ordered partitions of $D_f$, and so each term in 
Equation \eqref{eqn:mainf} has the form
\begin{align*}  c_fa_f & = \sum\limits_{B \in S_f}(-1)^{n-|B|}a_f \\ & = (-1)^{n-|D_f|} \sum\limits_{A\in D_f }(-1)^{|D_f|-|A|}a_f \\ & = (-1)^{n-|D_f|}c_{\overline{f}} a_f, \end{align*}
and the lemma follows.
\end{proof}

The number of ordered partitions of $[n]$ with $k$ parts is well known to be $k!S(n,k)$, where $S(n,k)$ are the Stirling numbers of the second kind. We obtain an important corollary to Lemma \ref{lemma:flattenorbits} from the following well-known result about Stirling numbers of the second kind.

\begin{lemma}\label{lemma:stirling}
The following identity holds for ordered partitions:
\[
\sum\limits_{k=0}^n(-1)^{n-k}k!S(n,k)=1.
\]
\end{lemma}

\begin{proof}
The result follows immediately upon setting $x=-1$ in the following identity on Stirling numbers of the second kind\cite[p.~35]{Sta12}:
\[
\sum\limits_{k=0}^nS(n,k)(x)_k=x^n.   \qedhere 
\]
\end{proof}

\begin{corollary} \label{corollary:permutation}
 If $f:[n]\to [n]$ is bijective, i.e., $f=\pi$ for some $\pi \in \Sym_n$,  then $c_f=\sgn(\pi)$. 
\end{corollary}

\begin{proof} Since $f$ is bijective, it consists only of cycles. Thus $D_f = 
C_f = \{\sigma_1, \ldots, \sigma_r \}$ and $|D_f|=r$.  By Lemma \ref{lemma:flattenorbits} and Lemma \ref{lemma:stirling} we see

\begin{align*} c_f & = (-1)^{n-|D_f|}\sum\limits_{A\in \D }(-1)^{|D_f|-|A|} \\& = (-1)^{n-r}\sum\limits_{B\vdash[r]}(-1)^{r-|B|} \\& = (-1)^{n-r}\sum\limits_{k=0}^n(-1)^{r-k}k!S(r,k) \\& = (-1)^{n-r}. \end{align*} We leave it to the reader to verify that $\sgn(\pi)=(-1)^{n-r}$.
\end{proof}

With Corollary \ref{corollary:permutation}, we have that the terrible expansion contains every term of the determinant with the 
correct coefficient.
It remains to show that whenever $f$ is not bijective, $c_f=0$. By Lemma~\ref{lemma:flattenorbits}, it suffices to show this for 
acyclic functions.

\section{The Poset of Ordered Partitions} \label{sec:partitions}

We next consider the poset of ordered partitions in order to show that the set $S_f$ has a nice structure when $f$ is acyclic.  This will allow us to eventually switch to a more geometric viewpoint.
 
Let $\Pn$ denote the poset of ordered partitions of the set $[n]$.
In Figure \ref{poset3}, we see the poset $\Pthree$ of ordered partitions on 3 elements. At the top of the poset $\Pn$, we have 
the $n!$ ordered partitions consisting of singletons. 
These partitions correspond bijectively with the elements of $\Sym_n$. Because of their importance later, 
we will refer to them as \itshape singleton partitions. \upshape 

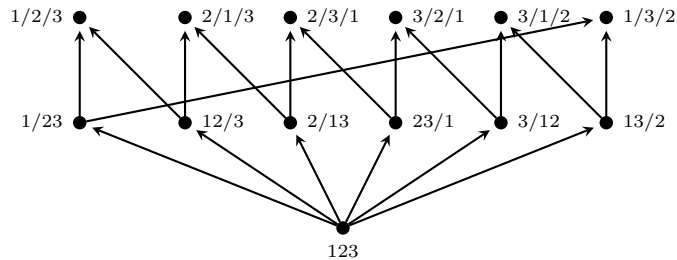
\begin{figure}[h] 
\begin{center}\begin{tikzpicture} 
[vertex/.style={circle,draw=black,fill=black,thick,inner sep=0pt,minimum size=1.5mm}, 
pre/.style={<-,shorten >=1pt,>=stealth,semithick}, 
post/.style={->,shorten >=2.5pt,>=stealth,thick},
noarrow/.style={-,thick}, scale=.7]

\node at (0,0) (123) [vertex] [label={below, font=\tiny}: 123] {};
\node at (-5,2) (1_23) [vertex] [label={left, font=\tiny}: {1/23}] {};
\node at (-3,2) (12_3) [vertex] [label={right, font=\tiny}: {12/3}] {};
\node at (-1,2) (2_13) [vertex] [label={right, font=\tiny}: {2/13}] {};
\node at (1,2) (23_1) [vertex] [label={right, font=\tiny}: {23/1}] {};
\node at (3,2) (3_12) [vertex] [label={right, font=\tiny}: {3/12}] {};
\node at (5,2) (13_2) [vertex] [label={right, font=\tiny}: {13/2}] {};

\node at (-5,4) (1_2_3) [vertex] [label={left, font=\tiny}: {1/2/3}] {};
\node at (-3,4) (2_1_3) [vertex] [label={right, font=\tiny}: {2/1/3}] {};
\node at (-1,4) (2_3_1) [vertex] [label={right, font=\tiny}: {2/3/1}] {};
\node at (1,4) (3_2_1) [vertex] [label={right, font=\tiny}: {3/2/1}] {};
\node at (3,4) (3_1_2) [vertex] [label={right, font=\tiny}: {3/1/2}] {};
\node at (5,4) (1_3_2) [vertex] [label={right, font=\tiny}: {1/3/2}] {};

\draw [post] (123) -- (1_23);
\draw [post] (123) -- (12_3);
\draw [post] (123) -- (2_13);
\draw [post] (123) -- (23_1);
\draw [post] (123) -- (3_12);
\draw [post] (123) -- (13_2);

\draw [post] (1_23) -- (1_2_3);
\draw [post] (12_3) -- (2_1_3);
\draw [post] (2_13) -- (2_3_1);
\draw [post] (23_1) -- (3_2_1);
\draw [post] (3_12) -- (3_1_2);
\draw [post] (13_2) -- (1_3_2);

\draw [post] (1_23) -- (1_3_2);
\draw [post] (12_3) -- (1_2_3);
\draw [post] (2_13) -- (2_1_3);
\draw [post] (23_1) -- (2_3_1);
\draw [post] (3_12) -- (3_2_1);
\draw [post] (13_2) -- (3_1_2);

\end{tikzpicture}
\caption{Poset of Ordered Partitions on 3 elements}
\label{poset3}
\end{center}
\end{figure}

Directly below a given ordered partition $B$ in $\Pn$ are ordered partitions formed by taking the union of two consecutive parts in $B$. 
For example, directly below the singleton partition 3/1/2/4 are the ordered partitions 13/2/4, 3/12/4, and 3/1/24. All ordered 
partitions under a singleton partition creates an $(n-1)$-cube. 
An example is given in Figure \ref{cube}.

An acyclic function can be viewed as a rooted forest, where the fixed points are the roots. An example is given in Figure \ref{forest}. 
Given an acyclic function $f$, if a path exists from $p$ to $q$, with $p$ closer to the root than $q$, then $f^k(q)=p$ for some $k$. Thus, 
$p\preceq q$, and so we say $f$ has the \itshape rule \upshape $p\preceq q$. In this way, each function $f$ stipulates a set of 
rules 
\[R_f=\left \{ p\preceq q \mid f^k(q) = p \text{ for } q,k\in[n] \right \}.\]
The following lemma and corollary will show that the set of ordered partitions $S_f$ has a nice structure in $\Pn$.

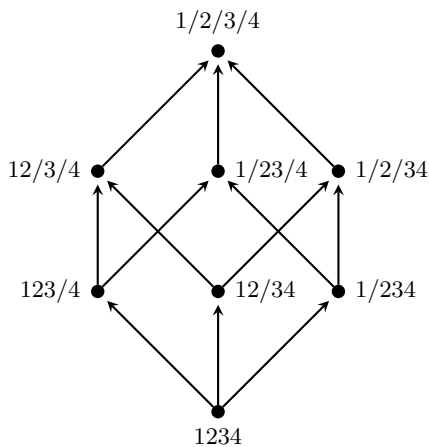
\begin{figure}[h] 
\begin{center}
\begin{tikzpicture} 
[vertex/.style={circle,draw=black,fill=black,thick,inner sep=0pt,minimum size=1.5mm}, 
redver/.style={circle,draw=red,fill=red,thick,inner sep=0pt,minimum size=1.5mm},
pre/.style={<-,shorten >=1pt,>=stealth,semithick}, 
post/.style={->,shorten >=2.5pt,>=stealth,thick},
noarrow/.style={-,thick},
redge/.style={-,thick,draw=red, fill=red}, scale=.4]

\node at ( 23,-4) (1) [vertex] [label={below, font=\footnotesize}: $1234$] {};
\node at ( 19,0) (2) [vertex] [label={left, font=\footnotesize}: {123/4}] {};
\node at ( 23,0) (3) [vertex] [label={right, font=\footnotesize}: {12/34}] {};
\node at ( 27,0) (5) [vertex] [label={right, font=\footnotesize}: {1/234}] {};
\node at ( 19,4) (6) [vertex] [label={left, font=\footnotesize}: {12/3/4}] {};
\node at ( 23,4) (10) [vertex] [label={right, font=\footnotesize}: {1/23/4}] {};
\node at ( 27,4) (15) [vertex] [label={right, font=\footnotesize}: {1/2/34}] {};
\node at ( 23,8) (30) [vertex] [label={above, font=\footnotesize}: {1/2/3/4}] {};

\draw [post] (1) -- (2);
\draw [post] (1) -- (3);
\draw [post] (1) -- (5);
\draw [post] (2) -- (6);
\draw [post] (2) -- (10);
\draw [post] (3) -- (6);
\draw [post] (3) -- (15);
\draw [post] (5) -- (10);
\draw [post] (5) -- (15);
\draw [post] (6) -- (30);
\draw [post] (10) -- (30);
\draw [post] (15) -- (30);

\end{tikzpicture}
\caption{The Cube Under the Singleton Partition 1/2/3/4}
\label{cube}
\end{center}
\end{figure}

\begin{figure}[h] 
\begin{center}\begin{tikzpicture} 
[vertex/.style={circle,draw=black,fill=black,thick,inner sep=0pt,minimum size=1.5mm}, 
pre/.style={<-,shorten >=1pt,>=stealth,semithick}, 
post/.style={->,shorten >=2.5pt,>=stealth,thick},
noarrow/.style={-,thick}, scale=.5]

\node at (2,2) (2) [vertex] [label={below, font=\small}: 2] {};
\node at (0,0) (1) [vertex] [label={left, font=\small}: 1] {};
\node at (4,0) (5) [vertex] [label={right, font=\small}: 5] {};

\node at (7,1) (7) [vertex] [label={below, font=\small}: 7] {};
\node at (9,3) (4) [vertex] [label={left, font=\small}: 4] {};
\node at (11,1) (6) [vertex] [label={right, font=\small}: 6] {};
\node at (11,-1) (3) [vertex] [label={right, font=\small}: 3] {};

\draw [post] (5) -- (2);
\draw [post] (1) -- (2);

\draw [post] (7) -- (4);
\draw [post] (3) -- (6);
\draw [post] (6) -- (4);

\path[->,every loop/.style={looseness=30}] (4) edge  [in=45,out=135,shorten >=2.5pt, shorten <=2.5pt, >=stealth,loop, thick, decoration={markings, mark=at position .999 with {\arrow[line width=2pt]{>}}},
    postaction={decorate}] node {} (); 
		
\path[->,every loop/.style={looseness=30}] (2) edge  [in=45,out=135,shorten >=2.5pt, shorten <=2.5pt, >=stealth,loop, thick, decoration={markings, mark=at position .999 with {\arrow[line width=2pt]{>}}},
    postaction={decorate}] node {} ();

\end{tikzpicture}
\caption{Acyclic Function Represented by a Rooted Forest}
\label{forest}
\end{center}
\end{figure}
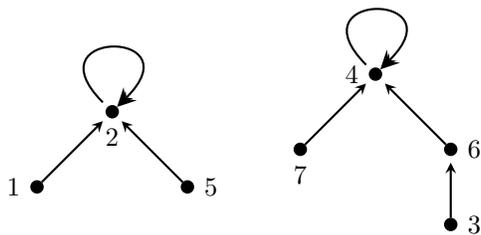

\begin{lemma} \label{lemma:party}
If a singleton partition $A$ in $\Pn$ satisfies the rules $R_f$ given by an 
acyclic function $f$ then every ordered partition $B\leq A$ 
in $\Pn$ also satisfies the rules $R_f$.
If an ordered partition $B$ in $\Pn$ satisfies the rules $R_f$ then there exists 
at least one singleton partition $A \geq B$ in $\Pn$ that satisfies the 
precedence rules 
$R_f$.
\end{lemma}

\begin{proof} The first statement is obvious.  If $A$ is a singleton partition with the rule $p \preceq q$ in $A$, and $B \leq A$ 
is an 
ordered partition in $\Pn$,
then the parts of $B$ are unions of consecutive parts of $A$. Hence $p \preceq q$ in $B$ as well.

The second statement is slightly less obvious.
Let $f$ be an acyclic function, 
and suppose that $B$ in $\Pn$ is a nonsingleton ordered partition that satisfies the rules $R_f$.
We must show there is a singleton partition $A \geq B$ above it in $\Pn$ that also satisfies $R_f$. Consider a part 
$\beta$ of $B$ that is not a singleton. If no pair of elements in $\beta$ has a rule associated with it, then the elements 
of $\beta$ can be ordered arbitrarily. 
Otherwise, there are elements of $\beta$ that have rules imposed on them. Consider the elements in the intersection of $\beta$ 
and a 
rooted tree associated with $f$. We order those elements by their distance from the root. (Those elements having the same 
distance from the root can be put in any order 
with respect to each other.) We do this for every rooted tree associated with $f$ to impose an order on all of $\beta$. 
Doing the same to each part will result in a singleton partition $A$ above $B$ satisfying the rules of $f$.
\end{proof}

\begin{corollary}\label{cubes}
The set of ordered partitions of $[n]$ satisfying the rules of an acyclic function is a union of $(n-1)$-cubes in $\Pn$.
\end{corollary}

\begin{proof} By the above lemma, we can account for all the ordered partitions by only considering the singleton partitions satisfying the acyclic function, and all the ordered partitions below them. The result follows since the ordered partitions below a singleton partition form an $(n-1)$-cube.
\end{proof}

Lemma \ref{lemma:party} and Corollary \ref{cubes} tell us that once we know which singleton partitions appear in $S_f$, then we 
know that $S_f$ is precisely those singletons and all the ordered partitions under them in $\Pn$. Because of their importance we 
will start to label the singleton partitions without slashes, e.g. $1/2/3/4\to 1234$, unless we need to distinguish them from the 
ordered partition with one part. In the next section, we turn our attention to a geometric object isomorphic to $\Pn$, the 
permutahedron.

\section{The Permutahedron} \label{sec:permutahedron}
 
In this section we show how the alternating sums giving $c_f$ when $f$ is acyclic are related to 
the Euler characteristic of the permutahedron and use this correspondence to show that $c_f = 0$.   
It is well-known that the poset of the ordered partitions is isomorphic to the face lattice of the permutahedron 
\cite[Fact 4.1]{Simion97}.
Specifically, each vertex on the permutahedron represents a singleton partition, the edges incident to a vertex represent the ordered partitions just 
below that singleton partition in the poset, the faces adjacent to those edges represent the ordered partitions just below again, and so on, 
until the permutahedron itself represents the ordered partition with one part at the bottom of the poset. 
Note that two vertices are adjacent if one can be obtained by a single swap of consecutive elements. For example, 315624 is adjacent to 351624. 
Figure \ref{perm3} shows the transformation from the poset on 3 element to the permutahedron on 3 elements, 
which 
in this case is a hexagon. Figure \ref{perm4} shows the poset on 4 elements as the ordinary permutahedron (truncated octahedron).

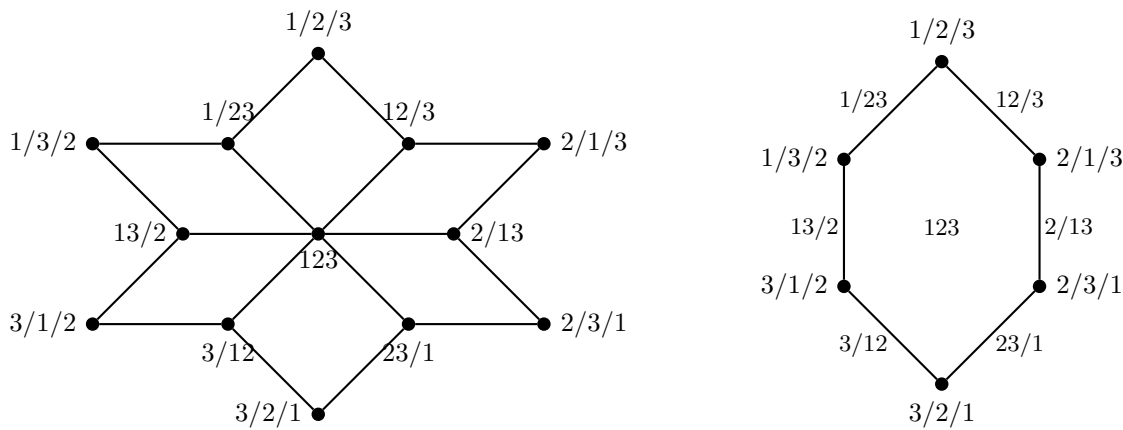
\begin{figure}[hbtp!] 
\begin{center}\begin{tikzpicture} 
[vertex/.style={circle,draw=black,fill=black,thick,inner sep=0pt,minimum size=1.5mm}, vertexb/.style={circle,draw=black,fill=black,thick,inner sep=0pt,minimum size=1.5mm}, 
pre/.style={<-,shorten >=1pt,>=stealth,semithick}, 
post/.style={->,shorten >=2.5pt,>=stealth,thick},
noarrow/.style={-,thick}, scale=.6]

\node at (0,0) (123) [vertex] [label={below, font=\small}: {123}] {};
\node at (0,-4) (3_2_1) [vertex] [label={left, font=\small}: {3/2/1}] {};
\node at (-2,-2) (3_12) [vertexb] [label={below, font=\small}: {3/12}] {};
\node at (-5,-2) (3_1_2) [vertex] [label={left, font=\small}: {3/1/2}] {};
\node at (-3,0) (13_2) [vertexb] [label={left, font=\small}: {13/2}] {};
\node at (0,4) (1_2_3) [vertex] [label={above, font=\small}: {1/2/3}] {};
\node at (-2,2) (1_23) [vertexb] [label={above, font=\small}: {1/23}] {};
\node at (-5,2) (1_3_2) [vertex] [label={left, font=\small}: {1/3/2}] {};

\node at (2,-2) (23_1) [vertexb] [label={below, font=\small}: {23/1}] {};
\node at (5,-2) (2_3_1) [vertex] [label={right, font=\small}: {2/3/1}] {};
\node at (3,0) (2_13) [vertexb] [label={right, font=\small}: {2/13}] {};
\node at (2,2) (12_3) [vertexb] [label={above, font=\small}: {12/3}] {};
\node at (5,2) (2_1_3) [vertex] [label={right, font=\small}: {2/1/3}] {};

\draw [noarrow] (123) -- (2_13);
\draw [noarrow] (123) -- (12_3);
\draw [noarrow] (123) -- (1_23);
\draw [noarrow] (123) -- (13_2);
\draw [noarrow] (123) -- (3_12);
\draw [noarrow] (123) -- (23_1);

\draw [noarrow] (2_13) -- (2_1_3);
\draw [noarrow] (2_1_3) -- (12_3);
\draw [noarrow] (12_3) -- (1_2_3);
\draw [noarrow] (1_2_3) -- (1_23);
\draw [noarrow] (1_23) -- (1_3_2);
\draw [noarrow] (1_3_2) -- (13_2);
\draw [noarrow] (13_2) -- (3_1_2);
\draw [noarrow] (3_1_2) -- (3_12);
\draw [noarrow] (3_12) -- (3_2_1);
\draw [noarrow] (3_2_1) -- (23_1);
\draw [noarrow] (23_1) -- (2_3_1);
\draw [noarrow] (2_3_1) -- (2_13);

\end{tikzpicture}
\hspace{.5in}
\begin{tikzpicture} 
[vertex/.style={circle,draw=black,fill=black,thick,inner sep=0pt,minimum size=1.5mm}, 
pre/.style={<-,shorten >=1pt,>=stealth,semithick}, 
post/.style={->,shorten >=2.5pt,>=stealth,thick},
noarrow/.style={-,thick}, scale=.65]

\node at (0,0) (3_2_1) [vertex] [label={below, font=\small}: {3/2/1}] {};
\node at (-2,2) (3_1_2) [vertex] [label={left, font=\small}: {3/1/2}] {};
\node at (2,2) (2_3_1) [vertex] [label={right, font=\small}: {2/3/1}] {};
\node at (-2,4.6) (1_3_2) [vertex] [label={left, font=\small}: {1/3/2}] {};
\node at (2,4.6) (2_1_3) [vertex] [label={right, font=\small}: {2/1/3}] {};
\node at (0,6.6) (1_2_3) [vertex] [label={above, font=\small}: {1/2/3}] {};

\draw [noarrow] (3_2_1) -- (3_1_2);
\draw [noarrow] (3_2_1) -- (2_3_1);
\draw [noarrow] (3_1_2) -- (1_3_2);
\draw [noarrow] (2_3_1) -- (2_1_3);
\draw [noarrow] (1_3_2) -- (1_2_3);
\draw [noarrow] (2_1_3) -- (1_2_3);

\node at (-1.6,.8) {\footnotesize 3/12};
\node at (1.6,.8) {\footnotesize 23/1};
\node at (-2.6,3.2) {\footnotesize 13/2};
\node at (2.6,3.2) {\footnotesize 2/13};
\node at (-1.6,5.8) {\footnotesize 1/23};
\node at (1.6,5.8) {\footnotesize 12/3};
\node at (0,3.2) {\footnotesize 123};

\end{tikzpicture}
\caption{Poset of Ordered Partitions on 3 elements: Top View (left) and as a Permutahedron (right)}
\label{perm3}
\end{center}
\end{figure}

\newcommand{\vertexsize}{1.2mm}
 
 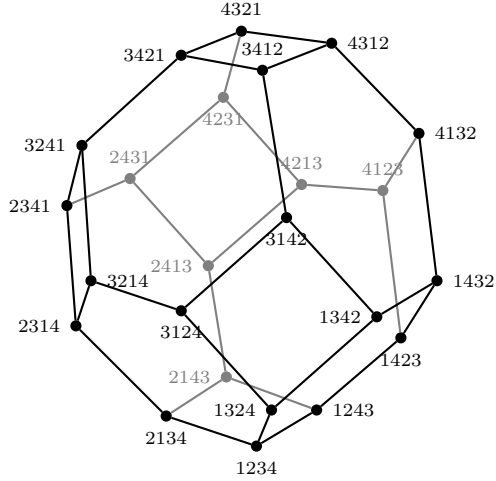
\begin{figure}[h]
 \begin{center}\begin{tikzpicture} 
 [vertex/.style={circle,draw=black,fill=black,thick,inner sep=0pt,minimum size=\vertexsize}, fadedvertex/.style={circle,draw=gray,fill=gray,thick,inner sep=0pt,minimum size=\vertexsize}, 
 pre/.style={<-,shorten >=1pt,>=stealth,semithick}, 
 post/.style={->,shorten >=2.5pt,>=stealth,thick},
 noarrow/.style={-,thick}, faded/.style={-,thick,gray},scale=.4] 
 \node at (0,0) (1234) [vertex] [label={below, font=\tiny}:$1234$] {};
 \node at ( -3,1) (2134) [vertex] [label={below, font= \tiny}:$2134$] {};
 \node at ( 0.5,1.2) (1324) [vertex] [label={left, font= \tiny}:$1324$] {};
 \node at ( 2,1.2) (1243) [vertex] [label={right, font= \tiny}:$1243$] {};
 \node at ( -1,2.3) (2143) [fadedvertex] [label={left, font=\tiny, gray}:$2143$] {};
 \node at ( -6,4) (3124) [vertex] [label={left, font=\tiny}:$2314$] {};
 \node at ( -2.5,4.5) (2314) [vertex] [label={below, font=\tiny}:$3124$] {};
 \node at ( 4,4.3) (1423) [vertex] [label={left, font=\tiny}:$1342$] {};
 \node at ( 4.8,3.6) (1342) [vertex] [label={below, font=\tiny}:$1423$] {};
 \node at ( -5.5,5.5) (3214) [vertex] [label={right, font=\tiny}:$3214$] {};
 \node at ( -1.6,6) (3142) [fadedvertex] [label={left, font=\tiny, gray}:$2413$] {};
 \node at ( 6,5.5) (1432) [vertex] [label={right, font=\tiny}:$1432$] {};
 \node at ( -6.3,8) (4123) [vertex] [label={left, font=\tiny}:$2341$] {};
 \node at ( -4.2,8.9) (4132) [fadedvertex] [label={above, font=\tiny,gray}:$2431$] {};
 \node at ( 1,7.6) (2413) [vertex] [label={below, font=\tiny}:$3142$] {};
 \node at ( 4.2,8.5) (2341) [fadedvertex] [label={above, font=\tiny,gray}:$4123$] {};
 \node at ( -5.8,10) (4213) [vertex] [label={left, font=\tiny}:$3241$] {};
 \node at ( 1.5,8.7) (3241) [fadedvertex] [label={above, font=\tiny,gray}:$4213$] {};
 \node at ( 5.4,10.4) (2431) [vertex] [label={right, font=\tiny}:$4132$] {};
 \node at ( -1.1,11.6) (4231) [fadedvertex] [label={below, font=\tiny,gray}:$4231$] {};
 \node at ( -2.5,13) (4312) [vertex] [label={left, font=\tiny}:$3421$] {};
 \node at ( 0.2,12.5) (3412) [vertex] [label={above, font=\tiny}:$3412$] {};
 \node at ( 2.5,13.4) (3421) [vertex] [label={right, font=\tiny}:$4312$] {};
 \node at ( -.5,13.8) (4321) [vertex] [label={above, font=\tiny}:$4321$] {};
 \draw [noarrow] (2134) -- (3124);
 \draw [faded] (2134) -- (2143);
 \draw [faded] (1243) -- (2143);
 \draw [noarrow] (1243) -- (1342);
 \draw [faded] (2143) -- (3142);
 \draw [faded] (3142) -- (4132);
 \draw [faded] (3142) -- (3241);
 \draw [faded] (4132) -- (4231);
 \draw [faded] (3241) -- (4231);
 \draw [faded] (2341) -- (3241);
 \draw [faded] (2341) -- (2431);
 \draw [faded] (1342) -- (2341);
 \draw [noarrow] (1342) -- (1432);
 \draw [noarrow] (4123) -- (4213);
 \draw [faded] (4123) -- (4132);
 \draw [faded] (4231) -- (4321);
 \draw [noarrow] (3412) -- (4312);
 \draw [noarrow] (3412) -- (3421);
 \draw [noarrow] (1234) -- (2134);
 \draw [noarrow] (1234) -- (1324);
 \draw [noarrow] (1234) -- (1243);
 \draw [noarrow] (1324) -- (2314);
 \draw [noarrow] (1324) -- (1423);
 \draw [noarrow] (3124) -- (3214);
 \draw [noarrow] (3124) -- (4123);
 \draw [noarrow] (2314) -- (3214);
 \draw [noarrow] (2314) -- (2413);
 \draw [noarrow] (1423) -- (2413);
 \draw [noarrow] (1423) -- (1432);
 \draw [noarrow] (3214) -- (4213);
 \draw [noarrow] (1432) -- (2431);
 \draw [noarrow] (2431) -- (3421);
 \draw [noarrow] (4213) -- (4312);
 \draw [noarrow] (4312) -- (4321);
 \draw [noarrow] (3421) -- (4321);
 \draw [noarrow] (2413) -- (3412);
 \end{tikzpicture}
 \caption{Poset of Ordered Partitions on 4 elements as a Permutahedron}
 \label{perm4}
 \end{center}
 \end{figure}

In Figure \ref{perm4}, we only label the singleton partitions at the vertices, but the labeling of the other ordered partitions 
would be similar to Figure \ref{perm3}.  

The permutahedron $\Pi_n$ is often defined as the convex hull of the points \[P_{\sigma}=(\sigma(1), \sigma(2), \ldots, 
\sigma(n))\] 
for every $\sigma\in\Sym_n$. It is a convex polytope, and in particular, it is contractible to a point.  
Thus, the permutahedron has Euler characteristic 1 \cite{Simion97}.
Since the ordered partitions are in bijection with the faces of the permutahedron, the alternating sum of the ordered partitions 
is precisely the Euler characteristic of the permutahedron, and this gives us a second proof of Lemma \ref{lemma:stirling}. 

We 
adopt a slightly different convention, relabeling the vertices of the $\Pi_n$ to 
$P_{\sigma^{-1}}$. We will denote this relabeled permutahedron by $\Pi_n'$. Figure \ref{2perm} 
shows $\Pi_3$ in $\mathbb{R}^3$, the relabeled $\Pi_n'$, 
and the correspondence between $x_1\leq x_2$ and $1\preceq 2$. Figure \ref{2perm} also shows the fact that $\Pi_n$, and thus 
$\Pi_n'$, is an $(n-1)$-dimensional object, since all the points lie in the hyperplane $x_1+x_2+\cdots+x_n=\dbinom{n+1}{2}$. In 
general $x_i\leq x_j$ in $\Pi_n$ corresponds to $i\preceq j$ in $\Pi_n'$.

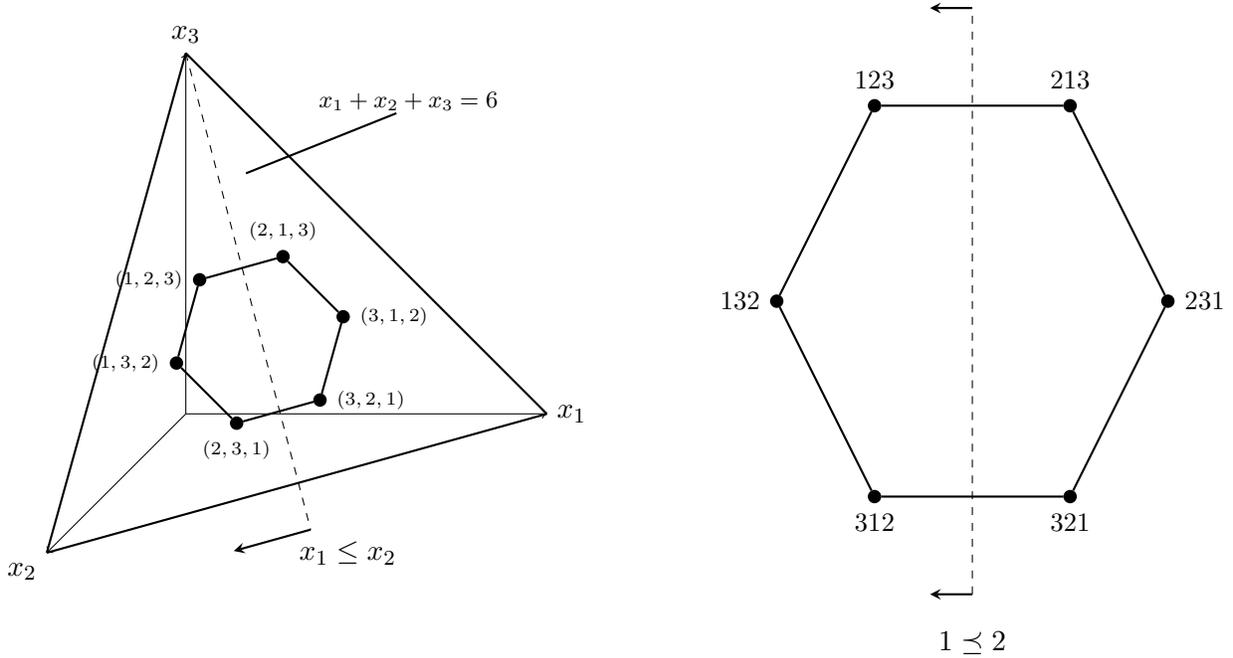
\begin{figure}
\begin{tikzpicture}
[vertex/.style={circle,draw=black,fill=black,thick,inner sep=0pt,minimum size=1.5mm}, 
pre/.style={<-,shorten >=1pt,>=stealth,semithick}, 
post/.style={->,shorten >=2.5pt,>=stealth,thick},
noarrow/.style={-,thick}, scale=.8]
  \draw [->] (0,0) -- (6,0,0) node [right] {$x_1$};
  \draw [->] (0,0) -- (0,6,0) node [above] {$x_3$};
  \draw [->] (0,0) -- (0,0,6) node [below left] {$x_2$};


  \draw[noarrow] (0,0,6) -- (0,6,0);
 \draw[noarrow] (6,0,0) -- (0,6,0);
  \draw[noarrow] (0,0,6) -- (6,0,0);
  
  \node at (1,3,2) (123) [vertex] [label={left, font=\tiny}:{$(1,2,3)$}] {};
  \node at (1,2,3) (132) [vertex] [label={left, font=\tiny}:{$(1,3,2)$}] {};
  \node at (2,3,1) (213) [vertex] [label={above, font=\tiny}:{$(2,1,3)$}] {};
  \node at (2,1,3) (231) [vertex] [label={below, font=\tiny}:{$(2,3,1)$}] {};
  \node at (3,2,1) (312) [vertex] [label={right, font=\tiny}:{$(3,1,2)$}] {};
  \node at (3,1,2) (321) [vertex] [label={right, font=\tiny}:{$(3,2,1)$}] {};
  
  \draw[noarrow] (123) -- (213);
  \draw[noarrow] (123) -- (132);
  \draw[noarrow] (213) -- (312);
  \draw[noarrow] (312) -- (321);
  \draw[noarrow] (321) -- (231);
  \draw[noarrow] (231) -- (132);
  
  \draw[noarrow] (1,4,0) -- (3.5,5,0);
  
  \node at (3.7,5.2,0) {\footnotesize $x_1+x_2+x_3=6$};
	
	\draw[dashed] (0,6,0) -- (4,0,5);
	\draw[post] (4,0,5) -- (3,0,6);
	\node at (5,0,6) {$x_1\leq x_2$};
  
  \node at (0,-4) {};

\end{tikzpicture}
\hspace{.5in}
\begin{tikzpicture}
[vertex/.style={circle,draw=black,fill=black,thick,inner sep=0pt,minimum size=1.5mm}, 
pre/.style={<-,shorten >=1pt,>=stealth,semithick}, 
post/.style={->,shorten >=2.5pt,>=stealth,thick},
noarrow/.style={-,thick}, scale=1.3]
\node at (1,0) (3_2_1) [vertex] [label={below, font=\small}: {321}] {};
\node at (-1,0) (3_1_2) [vertex] [label={below, font=\small}: {312}] {};
\node at (2,2) (2_3_1) [vertex] [label={right, font=\small}: {231}] {};
\node at (-2,2) (1_3_2) [vertex] [label={left, font=\small}: {132}] {};
\node at (1,4) (2_1_3) [vertex] [label={above, font=\small}: {213}] {};
\node at (-1,4) (1_2_3) [vertex] [label={above, font=\small}: {123}] {};

\draw [noarrow] (3_2_1) -- (3_1_2);
\draw [noarrow] (3_2_1) -- (2_3_1);
\draw [noarrow] (3_1_2) -- (1_3_2);
\draw [noarrow] (2_3_1) -- (2_1_3);
\draw [noarrow] (1_3_2) -- (1_2_3);
\draw [noarrow] (2_1_3) -- (1_2_3);

\draw[dashed] (0,-1) -- (0,5);

\draw[post] (0,-1) -- (-.5,-1);
\draw[post] (0,5) -- (-.5,5);

\node at (0,-1.5) {$1\preceq 2$};

\end{tikzpicture}
\caption{The permutahedra $\Pi_3$ (left) and $\Pi_3'$ (right)}
\label{2perm}
\end{figure}

\begin{lemma}\label{halfspace}
The singleton partitions in $\Pi_n'$ satisfying the precedence rule $i\preceq j$ are contained 
in the corresponding half-space $x_i\leq x_j$ in $\Pi_n$. 
\end{lemma}

\begin{proof} 
A permutation $\sigma^{-1}$ satisfies $\sigma^{-1}_i \preceq \sigma^{-1}_j$ precisely when $\sigma(i) \leq \sigma(j)$.
Hence we see that if the permutation (or singleton partition) $\sigma^{-1}$ satisfies the precedence relation $i \preceq j$ 
then the vertex  $(\sigma(1), \sigma(2), \ldots, \sigma(n) )$ is in the half-space $x_i\leq x_j$ and vice versa.  
\end{proof}

We are now ready to prove the main result of this section.

\begin{proposition} \label{prop:zerocoefficient}
If $f: [n] \to [n]$ is not bijective, then the coefficient $c_f=0$ is zero.  
\end{proposition}

\begin{proof}
Let $f:[n] \to [n]$ be an acyclic function.  Let $R_f = \{ f(j)=i \preceq j \}$ denote the set of precedence rules 
determined by $f$.
To each precedence rule $f(j) = i \preceq j$ in $R_f$ we can assign a half-space $H_{ij} := \{ (x_1,x_2,\ldots, x_n) \in  \mathbb{R}^n: x_{f(j)} = x_i \leq x_j \}$, 
and the intersection of these half-spaces with the permutahedron
$\Pi_n'$ defines a convex polytope, which we will denote by $\G_f$.  

The faces of $\G_f$ fall into two disjoint subsets: those faces that correspond to the ordered partitions in $S_f$ and those that do not.
Let $\Gamma_f$ denote the faces of $\G_f$ which correspond to elements of $S_f$ and let $\Delta_f$ denote the faces of $\G_f$ that do not.
The set of faces $\Delta_f$ are precisely the faces of $\G_f$ which lie entirely on the boundary of at least one half-space $x_i = x_j$ because they resulted from
intersecting $\Pi_n'$ with one of the half-spaces $H_{ij}$. 
Thus each face of $\Delta_f$ is a convex polytope, and we see that $\Delta_f$ is a union convex polytopes. 
Each of these convex polytopes contains the point $(x_1,x_2, \ldots, x_n)$ where $ x_1=x_2=\cdots= x_n$, so $\Delta_f$ is a contractible space.  
Hence $\Delta_f$ has Euler characteristic $\chi(\Delta_f) = 1$.  

Since $\G_f$ is a convex polytope and $\G_f = \Gamma_f \sqcup \Delta_f$ we see that
\[1=  \chi(\G_f) = \chi(\Gamma_f) + \chi( \Delta_f )  = 0+1.\]
It follows that $\chi(\Gamma_f) =0$.  
Since the summands of the alternating sum $\displaystyle \sum_{B \in S_f} (-1)^{n- |B|}$ correspond with the faces of $\Gamma_f$ we see that
$c_f = \displaystyle \sum_{B \in S_f} (-1)^{n- |B|} = \chi(\Gamma_f) =0$ as desired. 
\end{proof}

We end this section with an example demonstrating the proof of Proposition \ref{prop:zerocoefficient}.

\begin{example}
Figure \ref{permex2a} shows $\Pi_f$ with $\Gamma_f$ bold and $\Delta_f$ shaded for the terms $a_f=a_{11}a_{12}a_{13}$, which has the rules $1\preceq2$ and $1\preceq3$ and $a_f=a_{11}a_{12}a_{33}a_{34}$, which has the rules $1\preceq2$ and $3\preceq4$.

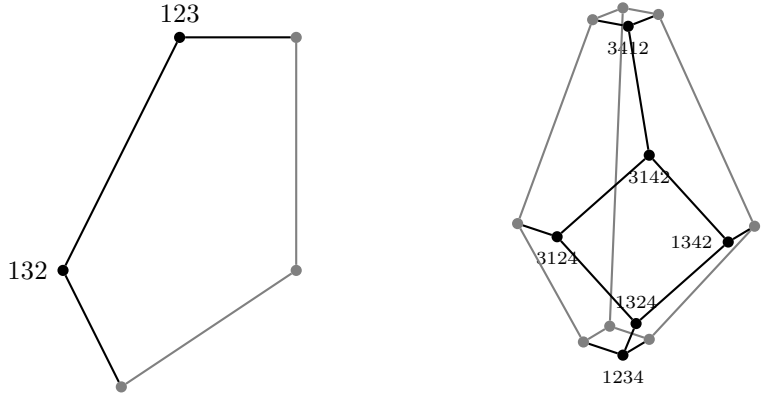
\begin{figure}[h]
  \begin{center}
	
	\begin{tikzpicture}
 [vertex/.style={circle,draw=black,fill=black,thick,inner sep=0pt,minimum size=\vertexsize}, bvertex/.style={circle,draw=blue,fill=blue,thick,inner sep=0pt,minimum size=\vertexsize}, fadedvertex/.style={circle,draw=gray,fill=gray,thick,inner sep=0pt,minimum size=\vertexsize}, 
  pre/.style={<-,shorten >=1pt,>=stealth,semithick}, 
  post/.style={->,shorten >=2.5pt,>=stealth,thick},
  noarrow/.style={-,thick}, faded/.style={-,thick,gray},scale=1.55] 
	
\node at (-2,2) (1_3_2) [vertex] [label={left, font=\small}: {132}] {};
\node at (-1,4) (1_2_3) [vertex] [label={above, font=\small}: {123}] {};
\node at (-1.5,1) (x_x_2) [fadedvertex] {};
\node at (0,4) (x_x_3) [fadedvertex] {};
\node at (0,2) (x_x_x) [fadedvertex] {};

\draw [noarrow] (x_x_3) -- (1_2_3);
\draw [noarrow] (1_2_3) -- (1_3_2);
\draw [faded] (x_x_3) -- (x_x_x);
\draw [noarrow] (1_3_2) -- (x_x_2);
\draw [faded] (x_x_2) -- (x_x_x);




\end{tikzpicture}
 \hspace{1in}
  \begin{tikzpicture} 
  [vertex/.style={circle,draw=black,fill=black,thick,inner sep=0pt,minimum size=\vertexsize}, bvertex/.style={circle,draw=blue,fill=blue,thick,inner sep=0pt,minimum size=\vertexsize}, fadedvertex/.style={circle,draw=gray,fill=gray,thick,inner sep=0pt,minimum size=\vertexsize}, 
  pre/.style={<-,shorten >=1pt,>=stealth,semithick}, 
  post/.style={->,shorten >=2.5pt,>=stealth,thick},
  noarrow/.style={-,thick}, faded/.style={-,thick,gray},scale=.35]

  \node at ( 0,0) (1234) [vertex] [label={below, font=\tiny}:$1234$] {};
  \node at ( 0.5,1.2) (1324) [vertex] [label={above, font= \tiny}:$1324$] {};
  \node at ( -2.5,4.5) (2314) [vertex] [label={below, font=\tiny}:$3124$] {};
  \node at ( 4,4.3) (1423) [vertex] [label={left, font=\tiny}:$1342$] {};
  \node at ( 1,7.6) (2413) [vertex] [label={below, font=\tiny}:$3142$] {};
  \node at ( 0.2,12.5) (3412) [vertex] [label={below, font=\tiny}:$3412$] {};
	\node at ( -1.5,0.5) (xx34) [fadedvertex] {};
  \node at ( 1,0.6) (12xx) [fadedvertex] {};
  \node at ( -4,5) (3xx4) [fadedvertex] {};
	\node at ( 5,4.9) (1xx2) [fadedvertex] {};
	\node at ( -1.15,12.75) (34xx) [fadedvertex] {};
	\node at ( 1.35,12.95) (xx12) [fadedvertex] {};
	\node at ( -0.5,1.1) (xxyy) [fadedvertex] {};
	\node at ( 0,13.2) (yyxx) [fadedvertex] {};

\draw [faded] (xx34) -- (3xx4);
\draw [faded] (34xx) -- (3xx4);
\draw [faded] (12xx) -- (1xx2);
\draw [faded] (xx12) -- (1xx2);

\draw [faded] (xx12) -- (yyxx);
\draw [faded] (34xx) -- (yyxx);
\draw [faded] (12xx) -- (xxyy);
\draw [faded] (xx34) -- (xxyy);

\draw [faded] (xxyy) -- (yyxx);

 \draw [noarrow] (3412) -- (34xx);
\draw [noarrow] (3412) -- (xx12);
  
\draw [noarrow] (1234) -- (xx34);
\draw [noarrow] (1234) -- (1324);
\draw [noarrow] (1234) -- (12xx);
  
 \draw [noarrow] (1324) -- (2314);
 \draw [noarrow] (1324) -- (1423);
  
  
\draw [noarrow] (2314) -- (3xx4);
\draw [noarrow] (2314) -- (2413);
  
 \draw [noarrow] (1423) -- (2413);
\draw [noarrow] (1423) -- (1xx2);

 \draw [noarrow] (2413) -- (3412);

  \end{tikzpicture}
  \caption{$a_f=a_{11}a_{12}a_{13}$ (left) and $a_f=a_{11}a_{12}a_{33}a_{34}$ (right)}
  \label{permex2a}
  \end{center}
  \end{figure}
\end{example}

\section{Multivariate Finite Operator Calculus} \label{sec:Hardstuff}

This terrible expansion of the determinant came from a conjecture about a transfer formula in \emph{multivariate finite operator 
calculus}
(MFOC). In this section we give a very brief overview of the objects of study in MFOC and the conjecture that gives this 
expansion. The interested reader is encouraged to read \cite{watanabe1984} for a more comprehensive description of this subject matter.

Let $k$ be a  field. Let $\{\textbf{e}_i\}_{1 \leq i \leq \ell} $ denote the standard 
$\ell$-dimensional 
basis of $k ^\ell$. 
The main objects of study in MFOC are polynomials $p\in k[x_1,\ldots,x_{\ell}]$ and operators $T\in k[[D_1,\ldots,D_{\ell}]]$, 
where $D_i$ is the partial derivative with respect to $x_i$. A sequence of polynomials 
$b_{\textbf{n}}(\textbf{x})=b_{n_1,\ldots,n_{\ell}}(x_1,...,x_{\ell})$ is called a Sheffer sequence if there is a set of operators 
$\textbf{B}=(B_1,\ldots,B_{\ell})$ with $B_i=D_iP_i$ where each $B_i:b_\textbf{n} \to 
b_{\textbf{n}-\textbf{e}_i}$ and each $P_i$ an invertible operator.  
Such a set of operators is called a delta $\ell$-tuple. The 
power series for an operator $T$ is written as

\begin{equation}
T=\sum\limits_{\textbf{n}\geq0}a_\textbf{n}\textbf{D}^\textbf{n}=\sum\limits_{n_1,\ldots,n_{\ell}\geq0}a_{n_1,\ldots,n_{\ell}}D_1^{n_1}\cdots D_{\ell}^{n_{\ell}}.
 \label{eqn:T} \end{equation}

These are all standard notations in any multivariate theory.
However, the following notation is not completely standard in MFOC.  
Given a subset $A\subseteq[{\ell}]$ we define $X_A:=\prod\limits_{i\in A}X_i$. 

The following is the Transfer Theorem from MFOC, and is essentially Theorem 1.3.6 in \cite{watanabe1984} with different notation.

\begin{theorem}[Transfer Theorem]
Let $\mathcal{J}$ denote the usual Jacobian matrix of a collection of polynomials.
Suppose $\textbf{B}=(B_1,B_2,\ldots,B_{\ell})$ is a delta $\ell$-tuple where $B_i=D_iP_i^{-1}$, then

\[
b_\textbf{n}(\textbf{x})=\textbf{P}^{\textbf{n}+\textbf{1}}{\mathcal J}(B_1,B_2,\ldots,B_{\ell})\dfrac{\textbf{x}^\textbf{n}}{\textbf{n}!}
\] is the basic sequence for $\textbf{B}$ written in terms of $\frac{\textbf{x}^\textbf{n}}{\textbf{n}!}$.

\end{theorem}

Within the Jacobian, we have Pincherle 
derivatives $\dfrac{\partial B_i}{\partial D_j}=B_i\theta_j-\theta_jB_i$, where $\theta_j:p\to x_jp$ is the $j$th umbral shift 
operator that does not commute with the delta operators. Thus, there are many ways this transfer formula could be expanded. The 
following conjecture (based on the examples provided below) gives one such way.

\begin{conjecture} \label{MainConjecture}

The basic sequence from the Transfer Theorem can also be calculated as

\[
b_\textbf{n}(\textbf{x})=\sum\limits_{B\vdash[\ell]}(-1)^{\ell-|B|}\left(\theta_{\beta}P_{\beta}\right)_B\dfrac{\textbf{x}^{\textbf{n}-\textbf{1}}}{\textbf{n}!},
\] where $B$ runs through all ordered partitions of $[\ell]$ and $\beta$ runs through the partitions of $B$.

\end{conjecture}

Because of our abuse of some notation, we give some examples.

\[
b_{m,n}(u,v)=\left(uP_1^mvP_2^n+vP_2^nuP_1^m-uvP_1^mP_2^n\right)\dfrac{u^{m-1}v^{n-1}}{m!n!}
\]

\footnotesize

\begin{eqnarray}
\notag b_{m,n,p}(a,b,c)&=&\left(aRbScT + aRcTbS + bSaRcT + bScTaR + cTaRbS + cTbSaR\right. \\\notag
 &&- aRbcST - bSacRT - cTabRS - abRScT - acRTbS - bcSTaR  \\\notag
 &&\left.+ abcRST\right)\dfrac{a^{m-1}b^{n-1}c^{p-1}}{m!n!p!} \\\notag
\end{eqnarray}

\normalsize

\[
B=(\{2\}, \{1,3\}) \quad \Rightarrow \quad \left(\theta_{\beta}P_{\beta}\right)_B=bSacRT
\]
The terrible expansion of the determinant comes from setting each $B_i=\textbf{D}^{\textbf{a}_i}=D_1^{a_{i1}}D_2^{a_{i2}}\cdots 
D_n^{a_{in}}$, or in other words, it is one term of the power series in Equation \eqref{eqn:T}.

\section{Conclusions and Open Questions}

We end with few open questions stemming from our work.
\begin{enumerate}
\item[\textbf{Q1.}] Now that Theorem \ref{thm:terrible} shows that Conjecture \ref{MainConjecture} is true for one term of an 
operator's power series, can Conjecture \ref{MainConjecture} be proven by a linearity argument?
\item[\textbf{Q2.}] Can the proof of Ryser's formula given by Horn and Johnson \cite{HJ} be modified to give another proof of 
Theorem \ref{thm:terrible}?
\item[\textbf{Q3.}] Can our proof of Theorem \ref{thm:terrible} be modified to prove Ryser's formula by using the 
topological/combinatorial properties of the cube instead of the permutahedon?
\end{enumerate}

\section{Acknowledgements}
The authors would like to thank Drs. Mohamed Omar, Pamela Harris, and Brian Johnson for helpful conversations during the writing 
of this paper.

\end{document}